\documentclass{amsart}
\usepackage{amssymb, amsmath, amsthm, graphics, comment, xspace, enumerate}
\newtheorem{thm}{Theorem}[section]

\newtheorem{lem}[thm]{Lemma}
\newtheorem{prop}[thm]{Proposition}
\theoremstyle{definition}
\newtheorem{defn}[thm]{Definition}

\theoremstyle{remark}
\newtheorem{rem}[thm]{Remark}
\numberwithin{equation}{section}

\begin{document}
\title[Degenerate $C$-distribution semigroups in lcs]{Degenerate $C$-distribution semigroups in locally convex spaces}
\author{Marko Kosti\' c}
\address{Faculty of Technical Sciences,
University of Novi Sad,
Trg D. Obradovi\' ca 6, 21125 Novi Sad, Serbia}
\email{marco.s@verat.net}

\author{Stevan Pilipovi\' c}
\address{Department for Mathematics and Informatics,
University of Novi Sad,
Trg D. Obradovi\' ca 4, 21000 Novi Sad, Serbia}
\email{pilipovic@dmi.uns.ac.rs}

\author{Daniel Velinov}
\address{Department for Mathematics, Faculty of Civil Engineering, Ss. Cyril and Methodius University, Skopje,
Partizanski Odredi
24, P.O. box 560, 1000 Skopje, Macedonia}
\email{velinovd@gf.ukim.edu.mk}

{\renewcommand{\thefootnote}{} \footnote{
\\ \text{  }  \ \    {\it Key words and phrases.} Degenerate $C$-distribution semigroups,
degenerate integrated $C$-semigroups, multivalued linear operators, locally convex spaces.
\\  \text{  }  \ \ This research is partially supported by grant 174024 of Ministry
of Science and Technological Development, Republic of Serbia.}}

\begin{abstract}
The main purpose of this paper is to investigate degenerate $C$-distribution semigroups in the setting of barreled
sequentially complete locally convex spaces.
In our approach, the infinitesimal generator of a degenerate $C$-distribution semigroup is a multivalued linear operator
and the regularizing operator $C$ is not necessarily injective. We provide a few important theoretical novelties, considering also exponential subclasses of degenerate $C$-distribution semigroups.
\end{abstract}
\maketitle

\section{Introduction and Preliminaries}

In an our recent paper \cite{C-ultra}, we have introduced and systematically analyzed
the classes of $C$-distribution semigroups and $C$-ultradistribution semigroups in locally convex spaces (cf. \cite{b42}-\cite{cizi}, \cite{fat1},
\cite{ki90}, \cite{k92}-\cite{knjigaho}, \cite{ku112}-\cite{li121},
\cite{me152}, \cite{1964}-\cite{w241}  and references cited therein).
The main aim of this paper is to continue this research by investigating the classes of degenerate $C$-distribution semigroups in the setting of barreled
sequentially complete locally convex spaces (cf. \cite{carol}, \cite{faviniyagi}, \cite{FKP},
\cite{me152} and \cite{svir-fedorov} for further information about well-posedness of abstract degenerate differential equations of first order). As mentioned in the abstract, we consider multivalued linear operators as
infinitesimal generators of such semigroups and allow the regularizing operator $C$ to be non-injective (cf. \cite{baskakov-chern}, \cite{ki90}, \cite{ku112}, \cite{isna-maiz} and
\cite{me152}-\cite{me155}  for the primary source of information on degenerate distribution semigroups in Banach spaces). In contrast to the analyses carried out in \cite[Section 2.2]{me152} and \cite[Section 3]{baskakov-chern}, we do not use any decomposition of the state space
$E.$

The organization of paper can be briefly described as follows. After explaining the basic things about vector-valued generalized function spaces necessary for our further work, in Section 2 we take a preliminary look at multivalued linear operators in locally convex spaces. In Section 3, we repeat some known facts and definitions about fractionally integrated
$C$-semigroups in locally convex spaces and their subgenerators (integral generators). Our main results are contained in Section 4, in which we analyze various themes concerning degenerate $C$-distribution semigroups in locally convex spaces and further generalize some of our recent results from \cite{C-ultra}.
The studies of differential and analytical properties of degenerate $C$-distribution semigroups as well as degenerate $q$-exponential $C$-distribution semigroups in locally convex spaces is out of the scope of this paper.

\subsection{Notation}
Unless specified otherwise,
we assume
that $E$ is a Hausdorff sequentially complete
locally convex space over the field of complex numbers, SCLCS for short.  Our standing assumption henceforth will be that
the state space $ E$ is barreled.
By
$L(E)$ we denote the space consisting of all continuous linear mappings from $E$ into
$E.$ The symbol $\circledast_{E}$ ($\circledast$, if there is no risk for confusion) denotes the fundamental system of seminorms which defines the topology of $E.$
The Hausdorff locally convex topology on $E^{\ast},$ the dual space\index{dual space} of $E,$
defines the
system $(|\cdot|_{B})_{B\in {\mathcal B}}$ of seminorms on
$E^{\ast},$ where $|x^{\ast}|_{B}:=\sup_{x\in
B}|\langle x^{\ast}, x \rangle |,$ $x^{\ast} \in E^{\ast},$ $B\in
{\mathcal B}.$ The bidual of $E$ is denoted by $E^{\ast \ast}.$ 
Recall, the polars of nonempty sets $M\subseteq E$ and $N\subseteq E^*$ are defined as follows
$M^{\circ}:=\{y\in E^*:|y(x)|\leq 1\text{ for all } x\in M\}$ and
$N^{\circ}:=\{x\in E:\;|y(x)|\leq 1\text{ for all } y\in N\}.$

Now we shall briefly described the main definitions and properties of vector-valued generalized function spaces used henceforth;
cf.
\cite{a43}, \cite{b42}, \cite{fat1},  \cite{k91}-\cite{k82},  \cite{knjigah}, \cite{kothe1},  \cite{ku113}, \cite{martinez}-\cite{meise}, \cite{me152}, \cite{pilip}-\cite{sch16}
 and references cited therein for more details. The Schwartz spaces of test functions $\mathcal{D}=C_0^{\infty}(\mathbb{R})$, ${\mathcal S}(\mathbb{R})$
and $\mathcal{E}=C^{\infty}(\mathbb{R})$
carry the usual topologies.
If $\emptyset \neq \Omega  \subseteq {\mathbb R},$ then the symbol $\mathcal{D}_{\Omega}$ denotes the subspace of $\mathcal{D}$ consisting of those functions $\varphi \in \mathcal{D}$ for which supp$(\varphi) \subseteq \Omega;$ $\mathcal{D}_{0}\equiv \mathcal{D}_{[0,\infty)}.$
The spaces
$\mathcal{D}'(E):=L(\mathcal{D},E)$,
$\mathcal{E}'(E):=L(\mathcal{E},E)$ and
$\mathcal{S}'(E):=L(\mathcal{S},E)$
are topologized in the usual way; the symbols
$\mathcal{D}'_{\Omega}(E)$,
$\mathcal{E}'_{\Omega}(E)$ and $\mathcal{S}'_{\Omega}(E)$ denote the subspaces of
$\mathcal{D}'(E)$, $\mathcal{E}'(E)$ and $\mathcal{S}'(E)$,
respectively, containing $E$-valued distributions
whose supports are contained in $\Omega ;$ $\mathcal{D}'_{0}(E)\equiv \mathcal{D}'_{[0,\infty)}(E)$, $\mathcal{E}'_{0}(E)\equiv \mathcal{E}'_{[0,\infty)}(E)$, $\mathcal{S}'_{0}(E)\equiv \mathcal{S}'_{[0,\infty)}(E).$
If $E={\mathbb C},$ then the above spaces are also denoted by $\mathcal{D}',$
$\mathcal{E}',$
$\mathcal{S}',$ $\mathcal{D}'_{\Omega},$
$\mathcal{E}'_{\Omega},$
$\mathcal{S}'_{\Omega},$
$\mathcal{D}_0'$,
$\mathcal{E}_0'$ and $\mathcal{S}_0'.$
By a regularizing
sequence in $\mathcal{D}$ we mean any sequence $(\rho_n)_{n\in {\mathbb N}}$ in
$\mathcal{D}_0$ for which there exists a function $\rho\in\mathcal{D}$ satisfying $\int_{-\infty}^{\infty}\rho
(t)\,dt=1,$ supp$(\rho)\subseteq [0,1]$ and $\rho_n(t)=n\rho(nt)$,
$t\in\mathbb{R},$ $n\in {\mathbb N}.$
If $\varphi$, $\psi:\mathbb{R}\to\mathbb{C}$ are
locally integrable functions, then we define the convolution products $\varphi*\psi$
and $\varphi*_0\psi$ by
$$
\varphi*\psi(t):=\int\limits_{-\infty}^{\infty}\varphi(t-s)\psi(s)\,ds\mbox{ and }
\varphi*_0
\psi(t):=\int\limits^t_0\varphi(t-s)\psi(s)\,ds,\;t\;\in\mathbb{R}.
$$
Notice that $\varphi*\psi=\varphi*_0\psi$, provided that supp$(\varphi)$ and supp$(\psi)$ are subsets of $[0,\infty).$
Given $\varphi\in\mathcal{D}$ and $f\in\mathcal{D}'$, or $\varphi\in\mathcal{E}$ and $f\in\mathcal{E}'$,
we define the convolution $f*\varphi$ by $(f*\varphi)(t):=f(\varphi(t-\cdot))$, $t\in\mathbb{R}$.
For $f\in\mathcal{D}'$, or for $f\in\mathcal{E}'$,
define $\check{f}$ by $\check{f}(\varphi):=f(\varphi (-\cdot))$, $\varphi\in\mathcal{D}$ ($\varphi\in\mathcal{E}$).
Generally, the convolution of two distribution $f$, $g\in\mathcal{D}'$, denoted by $f*g$,
is defined by $(f*g)(\varphi):=g(\check{f}*\varphi)$, $\varphi\in\mathcal{D}$.
If one of them belongs to ${\mathcal E}'({\mathbb R})$, then we know that $f*g\in\mathcal{D}'$ and supp$(f*g)\subseteq$supp$ (f)+$supp$(g)$.

Let $G$ be an $E$-valued distribution, and let $f : {\mathbb R} \rightarrow E$ be a locally integrable function.
As in the scalar-valued case, we define the $E$-valued distributions
$G^{(n)}$ ($n\in {\mathbb N}$) and $ hG$ ($h\in {\mathcal E}$); the regular $E$-valued distribution ${\mathbf f}$ is defined by ${\mathbf f}(\varphi):=\int_{-\infty}^{\infty}\varphi (t)  f(t) \, dt$ ($\varphi \in {\mathcal D}$).
We need the following auxiliary lemma whose proof can be deduced as in the scalar-valued case.

\begin{lem}\label{polinomi}
Suppose that $0<\tau \leq \infty,$ $n\in {\mathbb N}$. If $f : (0,\tau) \rightarrow E$ is a continuous function and
$$
\int \limits^{\tau}_{0}\varphi^{(n)}(t)f(t)\, dt=0,\quad \varphi \in {\mathcal D}_{(0,\tau)},
$$
then there exist elements $x_{0},\cdot \cdot \cdot, x_{n-1}$ in $E$ such that $f(t)=\sum^{n-1}_{j=0}t^{j}x_{j},$ $t\in (0,\tau).$
\end{lem}

Following L. Schwartz \cite{sch16}, it will be said that a distribution $G\in {\mathcal D}'(X)$ is of finite order on the interval $(-\tau,\tau)$ iff
there exist an integer $n\in {\mathbb N}_{0}$ and an $X$-valued continuous function $f : [-\tau,\tau] \rightarrow X$
such that
$$
G(\varphi)=(-1)^{n}\int^{\tau}_{-\tau}\varphi^{(n)}(t)f(t)\, dt,\quad \varphi \in {\mathcal D}_{(-\tau,\tau)},\,\, \tau>0.
$$
$G$ is of finite order iff $G$ is of finite order on any finite interval $(-\tau,\tau).$
In the case that $X$ is a quasi-complete (DF)-space, then it is well known
that each $X$-valued distribution is of finite order. \\

We refer the reader to \cite{C-ultra} for some characterizations of vector-valued distributions supported by a point.
If the space $E$ satisfies the property that any vector-valued distribution $G\in\mathcal{D}'(E)$ with supp$(G)\subseteq\{0\}$ can be represented as a finite sum of
vector-valued distributions of form $\delta^{(i)} \otimes x_{i}$, then we
say that $E$ is admissible.

\section[Multivalued linear operators]{Multivalued linear operators}

In this section, we present some definitions and properties of multivalued linear operators that will be necessary for our further work (cf. the monographs \cite{cross} by R. Cross and \cite{faviniyagi} by A. Favini-A. Yagi for more details on the subject).
The underlying SCLCS will be denoted
by $X$ and $Y;$ in the third section, we will coming back to our standing notation.

A multivalued map (multimap) ${\mathcal A} : X \rightarrow P(Y)$ is said to be a multivalued
linear operator (MLO) iff the following holds:
\begin{itemize}
\item[(i)] $D({\mathcal A}) := \{x \in X : {\mathcal A}x \neq \emptyset\}$ is a subspace of $X$;
\item[(ii)] ${\mathcal A}x +{\mathcal A}y \subseteq {\mathcal A}(x + y),$ $x,\ y \in D({\mathcal A})$
and $\lambda {\mathcal A}x \subseteq {\mathcal A}(\lambda x),$ $\lambda \in {\mathbb C},$ $x \in D({\mathcal A}).$
\end{itemize}
If $X=Y,$ then it is also said that ${\mathcal A}$ is an MLO in $X.$
An almost immediate consequence of the definition is that,
for every $x,\ y\in D({\mathcal A})$ and for every $\lambda,\ \eta \in {\mathbb C}$ with $|\lambda| + |\eta| \neq 0,$ we
have $\lambda {\mathcal A}x + \eta {\mathcal A}y = {\mathcal A}(\lambda x + \eta y).$ If ${\mathcal A}$ is an MLO, then ${\mathcal A}0$ is a linear manifold in $Y$
and ${\mathcal A}x = f + {\mathcal A}0$ for any $x \in D({\mathcal A})$ and $f \in {\mathcal A}x.$ Set $R({\mathcal A}):=\{{\mathcal A}x :  x\in D({\mathcal A})\}.$
The set ${\mathcal A}^{-1}0 = \{x \in D({\mathcal A}) : 0 \in {\mathcal A}x\}$ is called the kernel\index{multivalued linear operator!kernel}
of ${\mathcal A}$ and it is denoted by $N({\mathcal A}).$ The inverse ${\mathcal A}^{-1}$ of an MLO is defined by
$D({\mathcal A}^{-1}) := R({\mathcal A})$ and ${\mathcal A}^{-1} y := \{x \in D({\mathcal A}) : y \in {\mathcal A}x\}$.\index{multivalued linear operator!inverse}
It is easily seen that ${\mathcal A}^{-1}$ is an MLO in $X,$ as well as that $N({\mathcal A}^{-1}) = {\mathcal A}0$
and $({\mathcal A}^{-1})^{-1}={\mathcal A}.$ If $N({\mathcal A}) = \{0\},$ i.e., if ${\mathcal A}^{-1}$ is
single-valued, then ${\mathcal A}$ is said to be injective.

For any mapping ${\mathcal A}: X \rightarrow P(Y)$ we define $\check{{\mathcal A}}:=\{(x,y) : x\in D({\mathcal A}),\ y\in {\mathcal A}x\}.$ Then ${\mathcal A}$ is an MLO iff $\check{{\mathcal A}}$ is a linear relation in $X\times Y,$ ($(x,\lambda y_1)+(x,\lambda y_2)=(x,\lambda y_1+\lambda y_2)$, for $x\in X$ and $y\in Y$) i.e., iff $\check{{\mathcal A}}$ is a subspace of $X \times Y.$ Since no confusion
seems likely, we will sometimes identify ${\mathcal A}$ with its graph. \index{linear relation}

If ${\mathcal A},\ {\mathcal B} : X \rightarrow P(Y)$ are two MLOs, then we define its sum ${\mathcal A}+{\mathcal B}$ by $D({\mathcal A}+{\mathcal B}) := D({\mathcal A})\cap D({\mathcal B})$ and $({\mathcal A}+{\mathcal B})x := {\mathcal A}x +{\mathcal B}x,$ $x\in D({\mathcal A}+{\mathcal B}).$
It can be simply checked that ${\mathcal A}+{\mathcal B}$ is likewise an MLO.\index{multivalued linear operator!sum}

Let ${\mathcal A} : X \rightarrow P(Y)$ and ${\mathcal B} : Y\rightarrow P(Z)$ be two MLOs, where $Z$ is an SCLCS. The product of ${\mathcal A}$
and ${\mathcal B}$ is defined by $D({\mathcal B}{\mathcal A}) :=\{x \in D({\mathcal A}) : D({\mathcal B})\cap {\mathcal A}x \neq \emptyset\}$ and\index{multivalued linear operator!product}
${\mathcal B}{\mathcal A}x:=
{\mathcal B}(D({\mathcal B})\cap {\mathcal A}x).$ Then ${\mathcal B}{\mathcal A} : X\rightarrow P(Z)$ is an MLO and
$({\mathcal B}{\mathcal A})^{-1} = {\mathcal A}^{-1}{\mathcal B}^{-1}.$ The scalar multiplication of an MLO ${\mathcal A} : X\rightarrow P(Y)$ with the number $z\in {\mathbb C},$ $z{\mathcal A}$ for short, is defined by
$D(z{\mathcal A}):=D({\mathcal A})$ and $(z{\mathcal A})(x):=z{\mathcal A}x,$ $x\in D({\mathcal A}).$ It is clear that $z{\mathcal A}  : X\rightarrow P(Y)$ is an MLO and $(\omega z){\mathcal A}=\omega(z{\mathcal A})=z(\omega {\mathcal A}),$ $z,\ \omega \in {\mathbb C}.$

The integer powers of an MLO ${\mathcal A} :  X\rightarrow P(X)$ is defined recursively as follows: ${\mathcal A}^{0}=:I;$ if ${\mathcal A}^{n-1}$ is defined, set $
D({\mathcal A}^{n}) := \bigl\{x \in  D({\mathcal A}^{n-1}) : D({\mathcal A}) \cap {\mathcal A}^{n-1}x \neq \emptyset \bigr\},
$
and
$
{\mathcal A}^{n}x := \bigl({\mathcal A}{\mathcal A}^{n-1}\bigr)x =\bigcup_{y\in  D({\mathcal A}) \cap {\mathcal A}^{n-1}x}{\mathcal A}y,\quad x\in D( {\mathcal A}^{n}).
$
It is well known that $({\mathcal A}^{n})^{-1} = ({\mathcal A}^{n-1})^{-1}{\mathcal A}^{-1} = ({\mathcal A}^{-1})^{n}=:{\mathcal A}^{-n},$ $n \in {\mathbb N}$
and $D((\lambda-{\mathcal A})^{n})=D({\mathcal A}^{n}),$ $n \in {\mathbb N}_{0},$ $\lambda \in {\mathbb C}.$ Moreover,
if ${\mathcal A}$ is single-valued, then the above definitions are consistent with the usual definition of powers of ${\mathcal A}.$

If ${\mathcal A} : X\rightarrow P(Y)$ and ${\mathcal B} : X\rightarrow P(Y)$ are two MLOs, then we write ${\mathcal A} \subseteq {\mathcal B}$ iff $D({\mathcal A}) \subseteq D({\mathcal B})$ and ${\mathcal A}x \subseteq {\mathcal B}x$
for all $x\in D({\mathcal A}).$ Assume now that
a linear single-valued operator $S : D(S) \subseteq X \rightarrow Y$ has domain $D(S) = D({\mathcal A})$ and $S \subseteq {\mathcal A},$ where ${\mathcal A} : X\rightarrow P(Y)$
is an MLO. Then $S$ is called a\index{multivalued linear operator!section}
section of ${\mathcal A};$ if this is the case, we have ${\mathcal A}x = Sx + {\mathcal A}0,$ $x \in D({\mathcal A})$ and
$R({\mathcal A}) = R(S) + {\mathcal A}0.$

We say that an MLO operator  ${\mathcal A} : X\rightarrow P(Y)$ is closed if for any
nets $(x_{\tau})$ in $D({\mathcal A})$ and $(y_{\tau})$ in $Y$ such that $y_{\tau}\in {\mathcal A}x_{\tau}$ for all $\tau\in I$ we have that $\lim_{\tau \rightarrow \infty}x_{\tau}=x$ and
$\lim_{\tau \rightarrow \infty}y_{\tau}=y$ imply
$x\in D({\mathcal A})$ and $y\in {\mathcal A}x.$\index{multivalued linear operator!closed}

If ${\mathcal A} : X\rightarrow P(Y)$ is an MLO, then we define the adjoint ${\mathcal A}^{\ast}: Y^{\ast}\rightarrow P(X^{\ast})$\index{multivalued linear operator!adjoint}
of ${\mathcal A}$ by its graph
$$
{\mathcal A}^{\ast}:=\Bigl\{ \bigl( y^{\ast},x^{\ast}\bigr)  \in Y^{\ast} \times X^{\ast} :  \bigl\langle y^{\ast},y \bigr\rangle =\bigl \langle x^{\ast}, x\bigr \rangle \mbox{ for all pairs }(x,y)\in {\mathcal A} \Bigr\}.
$$
It is simply verified that ${\mathcal A}^{\ast}$
is a closed MLO, and that $ \langle y^{\ast},y \rangle =0$ whenever $y^{\ast}\in D({\mathcal A}^{\ast})$ and $y\in {\mathcal A}0.$

Concerning the integration of functions with values in SCLCS, we follow the approach of C. Martinez and M. Sanz \cite[pp. 99-102]{martinez}. Denote by $\Omega$ a locally compact\index{space!locally compact} and separable metric space\index{space!separable metric}
and  by $\mu$ a locally finite
Borel measure\index{measure!locally finite Borel} defined on $\Omega.$ Then the following fundamental lemma holds:

\begin{lem}\label{integracija-tricky}
Suppose that ${\mathcal A} : X\rightarrow P(Y)$ is a closed \emph{MLO}. Let $f : \Omega \rightarrow X$ and $g : \Omega \rightarrow Y$ be $\mu$-integrable, and let $g(x)\in {\mathcal A}f(x),$ $x\in \Omega.$ Then $\int_{\Omega}f\, d\mu \in D({\mathcal A})$ and $\int_{\Omega}g\, d\mu\in {\mathcal A}\int_{\Omega}f\, d\mu.$
\end{lem}

In \cite{FKP}, we have recently considered the  $C$-resolvent sets of MLOs in locally convex spaces
(where $C\in L(X)$ is injective, $C{\mathcal A}\subseteq {\mathcal A}C$).
The
$C$-resolvent set of an MLO ${\mathcal A}$ in $X,$ $\rho_{C}({\mathcal A})$ for short, is defined as the union of those complex numbers
$\lambda \in {\mathbb C}$ for which $R(C)\subseteq R(\lambda-{\mathcal A})$ and
$(\lambda - {\mathcal A})^{-1}C$ is a single-valued bounded operator on $X.$
The operator $\lambda \mapsto (\lambda -{\mathcal A})^{-1}C$ is called the $C$-resolvent of ${\mathcal A}$ ($\lambda \in \rho_{C}({\mathcal A})$). In  this paper, we analyze the general situation in which the operator
$C\in L(X)$ is not necessarily injective. Then the operator $(\lambda - {\mathcal A})^{-1}C$ is no longer single-valued, which additionally hinders our considerations and work.

\section[Fractionally integrated $C$-semigroups...]{Fractionally integrated $C$-semigroups in locally convex spaces}\label{maxx}

In this section, we will collect the most important facts and definitions about (degenerate) fractionally integrated $C$-semigroups in locally convex spaces.
Observe that we do not require the injectiveness of operator $C\in L(E)$. Denote by $g_{\alpha}(t)=\frac{t^{\alpha-1}}{\Gamma(\alpha)}$ for $t>0$.

\begin{defn}\label{2.1.1.1'} (\cite{catania})
Let $0<\alpha <\infty$ and $0<\tau \leq \infty .$
A strongly continuous operator family $(S_\alpha(t))_{t\in [0,\tau)}\subseteq L(E)$ is called a (local, if $\tau<\infty$) $\alpha$-times integrated $C$-semigroup
iff the following holds:
\begin{itemize}
\item[(i)] $S_\alpha(t)C=CS_\alpha(t)$, $t\in [0,\tau),$ and
\item[(ii)] For all $x\in E$ and $t,\ s\in [0,\tau)$ with $t+s\in [0,\tau),$ we have
\begin{align*}
S_\alpha(t)S_\alpha(s)x=\Biggl[\int_0^{t+s}-\int_0^t-\int_0^s\Biggr]
g_{\alpha}(t+s-r)S_\alpha(r)Cx\,dr.
\end{align*}
\end{itemize}
\end{defn}

By a $C$-regularized semigroup ($0$-times integrated $C$-regularized semigroup) we mean any strongly continuous operator family $(S_0(t)\equiv S(t))_{t\in [0,\tau)}\subseteq L(E)$
satisfying that $S(t)C=CS(t)$, $t\in [0,\tau)$ and $S(t+s)C=S(t)S(s)$ for all $t,\ s\in [0,\tau)$ with $t+s\in [0,\tau).$ A global $C$-regularized semigroup
$(S(t))_{t\geq 0}$ is said to be
entire analytic \index{$(a,k)$-regularized $C$-resolvent family!entire} iff, for every $x\in E,$ the mapping $t\mapsto S(t)x,$ $t\geq
0$ can be analytically extended to the whole complex plane. We refer the reader to \cite{l1} for the most important applications of non-degenerate $C$-regularized semigroups.

Let $0<\alpha \leq \infty$. In
the case $\tau=\infty ,$ $(S_{\alpha}(t))_{t\geq 0}$ is said to be
exponentially equicontinuous\index{$(a,k)$-regularized $C$-resolvent family!exponentially equicontinuous} (equicontinuous\index{$(a,k)$-regularized $C$-resolvent family!equicontinuous}) iff there exists
$\omega \in {\mathbb R}$ ($\omega =0$) such that the family $\{
e^{-\omega t} S_{\alpha}(t) : t\geq 0\}$ is equicontinuous. The integral generator $\hat{{\mathcal A}}$ of $(S_{\alpha}(t))_{t\in [0,\tau)}$ is defined by its graph
\[
\hat{{\mathcal A}}:=\Biggl\{(x,y)\in E\times E:S_\alpha(t)x-g_{\alpha} (t)Cx=\int\limits_0^tS_\alpha(s)y\,ds,\; t\in [0,\tau) \Biggr\}.
\]
The integral generator $\hat{{\mathcal A}}$ of $(S_{\alpha}(t))_{t\in [0,\tau)}$ is a closed MLO in $E.$
Furthermore, $\hat{{\mathcal A}}\subseteq C^{-1}\hat{{\mathcal A}}C$ in the MLO sense, with the equality in the case that
the operator $C$ is injective.

By a subgenerator of $(S_{\alpha}(t))_{t\in [0,\tau)}$ we mean any MLO ${\mathcal A}$ in $E$ satisfying the following two conditions:
\begin{itemize}
\item[(A)] $S_{\alpha}(t)x-g_{\alpha+1} (t)Cx=\int_0^tS_{\alpha}(s)y\,ds,\mbox{ whenever }t\in [0,\tau)\mbox{ and }y\in {\mathcal A}x.$
\item[(B)] For all $x\in E$ and $t\in [0,\tau),$  we have $\int^{t}_{0}S_{\alpha}(s)x\, ds \in D({\mathcal A})$ and
$S_{\alpha}(t)x-g_{\alpha+1}(t)Cx\in {\mathcal A}\int_0^t S_{\alpha}(s)x\,ds.$
\end{itemize}
If $(S_{\alpha}^{1}(t))_{t\in [0,\tau)}\subseteq  L(E),$ resp. $(S_{\alpha}^{2}(t))_{t\in [0,\tau)}\subseteq  L(E),$ is strongly continuous and satisfies only (B), resp. (A), then we say that
$(S_{\alpha}^{1}(t))_{t\in [0,\tau)},$ resp. $(S_{\alpha}^{2}(t))_{t\in [0,\tau)},$ is an $\alpha$-times integrated $C$-existence family with a subgenerator ${\mathcal A},$ resp.,
$\alpha$-times integrated $C$-uniqueness family with a subgenerator ${\mathcal A}.$

We denote by $\chi(S_{\alpha})$
the set consisting of all subgenerators of the $\alpha$-times integrated $C$-semigroup $(S_{\alpha}(t))_{t\in [0,\tau)}.$
It is well known that $\chi(S_{\alpha})$
can have infinitely many elements; if ${\mathcal A}\in \chi(S_{\alpha})$, then
${\mathcal A}\subseteq \hat{{\mathcal A}}.$ In general, the set $\chi(S_{\alpha})$ can be empty and the integral generator of $(S_{\alpha}(t))_{t\in [0,\tau)}$
need not be a subgenerator of $(S_{\alpha}(t))_{t\in [0,\tau)}$ in the case that $\tau <\infty.$ In global case, the integral generator $\hat{{\mathcal A}}$ of $(S_{\alpha}(t))_{t\geq 0}$ is always its subgenerator. If ${\mathcal A}$ is a closed subgenerator of $(S_{\alpha}(t))_{t\in [0,\tau)},$ defined locally  or globally, then we know that $C{\mathcal A}\subseteq {\mathcal A}C,$ $\hat{{\mathcal A}}\subseteq C^{-1}{\mathcal A}C$ and that the injectivity of $C$
implies $\hat{{\mathcal A}}= C^{-1}{\mathcal A}C.$ Suppose that $C$ is injective
and ${\mathcal A}$ is an MLO. Then there exists at most one $\alpha$-times integrated $C$-semigroup $(S_\alpha(t))_{t\in [0,\tau)}$ which do have ${\mathcal A}$ as a subgenerator
(\cite{catania}).

\section{The basic properties of degenerate $C$-distribution semigroups in locally convex spaces}
Throughout this section, we assume that $C\in L(E)$ is not necessarily injective operator. Since $E$ is barreled, the uniform boundedness principle \cite[p. 273]{meise} implies that each ${\mathcal G}\in {\mathcal D}'(L(E))$ is boundedly equicontinuous, i.e., that for every $p\in \circledast$ and for every bounded subset $B$ of ${\mathcal D}$, there exist $c>0$ and
$q\in \circledast$ such that
$
p({\mathcal G}(\varphi)x)\leq cq(x),\ \varphi \in B, \ x\in E.
$

We start this section by introducing the following definition.

\begin{defn}\label{cuds}
Let $\mathcal{G}\in\mathcal{D}_0'(L(E))$ satisfy $C\mathcal{G}=\mathcal{G}C.$
Then it is said that $\mathcal{G}$ is a pre-(C-DS) iff the following holds:
\[\tag{C.S.1}
\mathcal{G}(\varphi*_0\psi)C=\mathcal{G}(\varphi)\mathcal{G}(\psi),\quad \varphi,\;\psi\in\mathcal{D}.
\]
If, additionally,
\[\tag{C.S.2}
\mathcal{N}(\mathcal{G}):=\bigcap_{\varphi\in\mathcal{D}_0}N(\mathcal{G}(\varphi))=\{0\},
\]
then $\mathcal{G}$ is called a $C$-distribution semigroup, (C-DS) in short.
A pre-(C-DS) $\mathcal{G}$ is called dense iff
\[\tag{C.S.3}
\mathcal{R}(\mathcal{G}):=\bigcup\limits_{\varphi\in\mathcal{D}_0}R(\mathcal{G}(\varphi))
\text{ is dense in } E.
\]
\end{defn}

If $C=I,$ then we also write pre-(DS),(DS), instead of pre-(C-DS), (C-DS).

Suppose that $\mathcal{G}$ is a pre-(C-DS). Then
$\mathcal{G}(\varphi)\mathcal{G}(\psi)=\mathcal{G}(\psi)\mathcal{G}(\varphi)$ for all $\varphi,\,\psi\in\mathcal{D}$,
and $\mathcal{N}(\mathcal{G})$ is a closed subspace of $E$.

The structural characterization of a pre-(C-DS) $\mathcal{G}$
on its kernel space
$\mathcal{N}(\mathcal{G})$ is described in the following theorem (cf.  \cite[Proposition 3.1.1]{knjigah} and the proofs of \cite[Lemma 2.2]{ku112}, \cite[Proposition 3.5.4]{knjigah}).

\begin{thm}\label{delta-point}
Let $\mathcal{G}$ be a pre-$($C-DS$)$, and let the space $L(\mathcal{N}(\mathcal{G}))$ be admissible.
Then, with $N=\mathcal{N}(\mathcal{G})$ and $G_1$ being the restriction of $\mathcal{G}$ to $N$ $(G_1=\mathcal{G}_{|N})$,
we have:
There exists an integer $m\in {\mathbb N}$ for which there exist unique operators $T_0$, $T_1,\dots,T_m\in L(\mathcal{N}(\mathcal{G}))$ commuting with $C$ so that
$G_1=\sum_{j=0}^m\delta^{(j)}\otimes T_j$, $T_iC^i=(-1)^iT_0^{i+1}$, $0\leq i \leq m-1$ and
$T_0T_m=T_0^{m+2}=0$.
\end{thm}

Let $\mathcal{G}\in\mathcal{D}_0'(L(E))$ and let $T\in\mathcal{E}_0'$
i.e., $T$ is a scalar-valued distribution with compact support contained in $[0,\infty)$.
Define
\[
G(T):=\Bigl\{(x,y) \in E\times E : \mathcal{G}(T*\varphi)x=\mathcal{G}(\varphi)y\;\mbox{ for all }\;\varphi\in\mathcal{D}_0 \Bigr\}.
\]
Then  it can be easily seen that $G(T)$ is a closed MLO; furthermore, if $\mathcal{G}\in\mathcal{D}_0'(L(E))$ satisfy (C.S.2), then $G(T)$ is a closed linear operator.
Assuming that the regularizing operator $C$ is injective, definition of $G(T)$ can be equivalently introduced by replacing the set
$\mathcal{D}_0$ with the set $\mathcal{D}_{[0,\epsilon)}$ for any $\epsilon>0.$ In general case,
for every $\psi\in\mathcal{D}$, we have $\psi_+:=\psi\mathbf{1}_{[0,\infty)}\in\mathcal{E}_0'$, where
$\mathbf{1}_{[0,\infty)}$ stands for the characteristic function of $[0,\infty),$ so that the definition of $G(\psi_+)$ is clear.
We define the (infinitesimal) generator of a pre-(C-DS) $\mathcal{G}$ by ${\mathcal A}:=G(-\delta')$ (cf. \cite{C-ultra} for more details about non-degenerate case, and \cite[Definition 3.4]{baskakov-chern} and \cite{ki90} for some other approaches
used in degenerate case). Then $\mathcal{N}(\mathcal{G}) \times \mathcal{N}(\mathcal{G}) \subseteq {\mathcal A}$ and $\mathcal{N}(\mathcal{G}) = {\mathcal A}0,$ which simply implies that
${\mathcal A}$ is single-valued iff (C.S.2)  holds. If this is the case, then we also have that the operator $C$ must be injective:
Suppose that $Cx=0$ for some $x\in E.$ By (C.S.1), we get that $\mathcal{G}(\varphi)\mathcal{G}(\psi)x=0,$ $\varphi,\;\psi\in\mathcal{D}.$ In particular, $\mathcal{G}(\psi)x\in {\mathcal N}({\mathcal G})=\{0\}$ so that $\mathcal{G}(\psi)x= 0,$ $\psi\in\mathcal{D}.$ Hence,  $x\in {\mathcal N}({\mathcal G})=\{0\}$ and therefore $x=0.$

Further on, if $\mathcal{G}$ is a pre-(C-DS), $T\in\mathcal{E}_0'$  and $\varphi\in\mathcal{D}$,
then ${\mathcal G}(\varphi)G(T)\subseteq G(T)\mathcal{G}(\varphi)$, $CG(T)\subseteq G(T)C$
and $\mathcal{R}(\mathcal{G})\subseteq D(G(T))$.
If $\mathcal{G}$ is a pre-(C-DS) and $\varphi$, $\psi\in\mathcal{D}$,
then the assumption $\varphi(t)=\psi(t)$, $t\geq 0$, implies $\mathcal{G}(\varphi)=\mathcal{G}(\psi)$.
As in the Banach space case, we can prove the following (cf. \cite[Proposition 3.1.3, Lemma 3.1.6]{knjigah}): Suppose that $\mathcal{G}$ is a pre-(C-DS). Then $(Cx,\mathcal{G}(\psi)x)\in G(\psi_+)$, $\psi\in\mathcal{D},$ $x\in E$  and ${\mathcal A}\subseteq C^{-1}{\mathcal A}C,$ while $ C^{-1}{\mathcal A}C={\mathcal A}$ provided that $C$ is injective.
Furthermore, the following holds:

\begin{prop}\label{isto}
Let ${\mathcal G}$ be a pre-(C-DS), $S$, $T\in\mathcal{E}'_0$, $\varphi\in\mathcal{D}_0$, $\psi\in\mathcal{D}$
 and $x\in E$.
Then we have:
\begin{itemize}
\item[(i)] $(\mathcal{G}(\varphi)x$, $\mathcal{G}(\overbrace{T*\cdots*T}^m*\varphi)x)\in G(T)^m$, $m\in\mathbb{N}$.
\item[(ii)] $G(S)G(T)\subseteq G(S*T)$ with $D(G(S)G(T))=D(G(S*T))\cap D(G(T))$, and $G(S)+G(T)\subseteq G(S+T)$.
\item[(iii)] $(\mathcal{G}(\psi)x$, $\mathcal{G}(-\psi^{\prime})x-\psi(0)Cx)\in G(-\delta')$.
\item[(iv)] If $\mathcal{G}$ is dense, then its generator is densely defined.
\end{itemize}
\end{prop}

The assertions (ii)-(vi) of \cite[Proposition 3.1.2]{knjigah} can be reformulated for pre-(C-DS)'s in locally convex spaces; here it is only worth noting that the reflexivity of state space $E$ implies that the spaces $E^*$ and $E^{**}=E$ are both barreled and sequentially complete:

\begin{prop}\label{kuki}
Let $\mathcal{G}$ be a pre-(C-DS). 
 Then the following holds:
\begin{itemize}
\item[(i)] $C(\overline{\langle\mathcal{R}(\mathcal{G})\rangle})\subseteq\overline{\mathcal{R}(\mathcal{G})}$,
where $\langle\mathcal{R}(\mathcal{G})\rangle$
denotes the linear span of $\mathcal{R}(\mathcal{G})$.
\item[(ii)] Assume $\mathcal{G}$ is not dense and
$\overline{C\mathcal{R}(\mathcal{G})}=\overline{\mathcal{R}(\mathcal{G})}$.
Put $R:=\overline{\mathcal{R}(\mathcal{G})}$ and $H:=\mathcal{G}_{|R}$.
Then $H$ is a dense pre-($C_1$-DS) on $R$ with $C_1=C_{|R}$.
\item[(iii)] The dual $\mathcal{G}(\cdot)^*$ is a pre-($C^*$-DS) on $E^*$
and $\mathcal{N}(\mathcal{G}^*)=\overline{\mathcal{R}(\mathcal{G})}^{\circ}$.
\item[(iv)] If $E$ is reflexive,
then $\mathcal{N}(\mathcal{G})=\overline{\mathcal{R}(\mathcal{G}^*)}^{\circ}$.
\item[(v)]
The $\mathcal{G}^*$ is a ($C^*$-DS) in $E^*$ iff $\mathcal{G}$ is a dense pre-(C-DS).
If $E$ is reflexive, then $\mathcal{G}^*$ is a dense pre-($C^*$-DS)  in $E^*$ iff $\mathcal{G}$ is a (C-DS).
\end{itemize}
\end{prop}

The following proposition has been recently proved in \cite{C-ultra} in the case that the operator $C$ is injective (cf. \cite[Proposition 2]{ki90} for a pioneering result in
this direction). The  argumentation contained in \cite{C-ultra} shows that the injectivity of $C$ is superfluous:

\begin{prop}\label{kisinski}
Suppose that ${\mathcal G}\in {\mathcal D}^{\prime}_{0}(L(E))$ and ${\mathcal G}(\varphi)C=C{\mathcal G}(\varphi),$ $\varphi \in {\mathcal D}$.
Then ${\mathcal G}$ is a pre-(C-DS)  iff
\begin{align*}
{\mathcal G}\bigl(\varphi^{\prime}\bigr){\mathcal G}(\psi)-{\mathcal G}(\varphi){\mathcal G}\bigl(\psi^{\prime}\bigr)=\psi(0){\mathcal G}(\varphi)C-\varphi(0){\mathcal G}
(\psi)C,\quad \varphi,\ \psi \in {\mathcal D}.
\end{align*}
\end{prop}

In \cite{C-ultra}, we have recently proved that
every (C-DS) in locally convex space is uniquely determined by its generator.
Contrary to the single-valued case, different pre-(C-DS)'s
can have the same generator. To see this, we can employ \cite[Example 2.3]{ku112}: Let $C=I,$ $E$ is a Banach space and $T\in L(E)$
is nilpotent of order $n\geq 2.$ Then the pre-(C-DS)'s ${\mathcal G}_{1}(\cdot)\equiv \sum_{i=0}^{n-2}\cdot^{(i)}(0)T^{i+1}$
and ${\mathcal G}_{2}(\cdot)\equiv 0$ have the same generator ${\mathcal A}\equiv E\times E.$

In Theorem \ref{lokal-int-C} and Theorem \ref{lokal-int-C-prim}, we clarify connections between degenerate $C$-distribution semigroups and degenerate local integrated $C$-semigroups. For the proof of first theorem, we need some preliminaries from our previous research study of distribution cosine functions (see e.g. \cite[Section 3.4]{knjigah}):
Let $\eta\in\mathcal{D}_{[-2,-1]}$ be a fixed test function satisfying $\int_{-\infty}^{\infty}\eta (t)\,dt=1$.
Then, for every fixed $\varphi\in\mathcal{D}$, we define $I(\varphi)$ as follows
$$
I(\varphi)(x):=\int\limits_{-\infty}^x
\Biggl[\varphi(t)-\eta(t)\int\limits_{-\infty}^{\infty}\varphi(u)\,du\Biggr]\,dt,
\;\;x\in\mathbb{R}.
$$
It can be simply verified that, for every $\varphi\in\mathcal{D}$ and $n\in {\mathbb N},$ we have $I(\varphi)\in\mathcal{D}$, $I^{n}(\varphi^{(n)})=\varphi ,$
$\frac{d}{dx}I(\varphi)(x)=\varphi(x)-\eta(x)\int_{-\infty}^{\infty}\varphi(u)\,du$, $x\in\mathbb{R}$ as well as that, for every $\varphi\in\mathcal{D}_{[a,b]}$  ($-\infty<a<b<\infty$), we have:
$
\text{supp}( I(\varphi))\subseteq[\min(-2,a),\max(-1,b)].
$
This simply implies that, for every $\tau>2,$ $-1<b<\tau$  and for every $m,\ n\in {\mathbb N}$ with $m\leq n,$ we have:
\begin{align}\label{jednazba-zaat}
I^{n}\bigl(\mathcal{D}_{(-\tau,b]}\bigr)\subseteq \mathcal{D}_{(-\tau,b]}\mbox{ and }
\frac{d^{m}}{dx^{m}}I^{n}(\varphi)(x)=I^{m-n}\varphi(x),\quad \varphi \in {\mathcal D},\ x\geq 0,
\end{align}
where $I^{0}\varphi:=\varphi,$ $\varphi \in {\mathcal D}.$

Now we are ready to show the following extension of \cite[Proposition 4.3 a)]{ku112} ($E$ is a Banach space, $C=I$), given here with a different proof.

\begin{thm}\label{lokal-int-C}
Let $\mathcal{G}$ be a pre-(C-DS) generated by ${\mathcal A}$, and let $\mathcal{G}$ be of finite order.
Then, for every $\tau>0$, there exist a number $n_{\tau}\in\mathbb{N}$
and a local $n_{\tau}$-times integrated $C$-semigroup $(S_{n_{\tau}}(t))_{t\in [0,\tau)}$  such that
\begin{align}\label{utf-88}
{\mathcal G}(\varphi)x=(-1)^{n_{\tau}}\int \limits^{\infty}_{0}\varphi^{(n_{\tau})}(s)S_{n_{\tau}}(s)x\, dt,\quad \varphi \in {\mathcal D}_{(-\tau ,\tau)},\ x\in E.
\end{align}
Furthermore, $(S_{n_{\tau}}(t))_{t\in [0,\tau)}$ is an $n_{\tau}$-times integrated $C$-existence family with a subgenerator ${\mathcal A},$ and the admissibility of space $L(\mathcal{N}(\mathcal{G}))$ implies that
$S_{n_{\tau}}(t)x= 0$, $t\in [0,\tau)$ for some $x\in {\mathcal N}({\mathcal G}) $ iff $T_{i}x=0$ for
$0\leq i \leq n_{\tau}-1;$ see \emph{Theorem \ref{delta-point}(i)} with $m\geq n_{\tau}-1$.
\end{thm}

\begin{proof}
Let $\tau>2$ and $\rho \in {\mathcal D}_{[0,1]}$ with $\int \rho \, dm =1$ be fixed.
Set $\rho_{n}(\cdot):=n\rho (n\cdot),$ $n\in {\mathbb N}.$ Then, for every $t\in [0,\tau),$
the sequence $\rho_{n}^{t}(\cdot):=\rho_{n}(\cdot -t)$ converges to $\delta_{t}$ as $n\rightarrow +\infty$ (in the space of scalar-valued distributions).
Since $\mathcal{G}\in {\mathcal D}^{\prime}_{0}(L(E))$ and $\mathcal{G}$ is of finite order, we know that there exist
a number $n_{\tau}\in\mathbb{N}$
and a strongly continuous operator family $(S_{n_{\tau}}(t))_{t\in [0,\tau)}\subseteq L(E)$  such that (\ref{utf-88})
holds good. We will first prove that $(S_{n_{\tau}}(t))_{t\in [0,\tau)}$ is a local $n_{\tau}$-times integrated $C$-existence family
commuting with $C$ and having ${\mathcal A}$ as a subgenerator. In order to do that, observe that the commutation of $\mathcal{G}(\cdot)$ and $C$ yields
\begin{align*}
\int \limits^{\infty}_{0}\varphi^{(n_{\tau})}(s)CS_{n_{\tau}}(s)x\, dt=\int \limits^{\infty}_{0}\varphi^{(n_{\tau})}(s)S_{n_{\tau}}(s)Cx\, dt,\quad \varphi \in {\mathcal D}_{(-\tau ,\tau)},\ x\in E.
\end{align*}
Plugging $ \varphi=I^{n_{\tau}}(\rho_{n}^{t})$ in this expression (cf. also (\ref{jednazba-zaat})), we get that
\begin{align*}
\int \limits^{\infty}_{0}\rho_{n}^{t}(s)CS_{n_{\tau}}(s)x\, dt=\int \limits^{\infty}_{0}\rho_{n}^{t}(s)S_{n_{\tau}}(s)Cx\, dt,\quad \varphi \in {\mathcal D}_{(-\tau ,\tau)},\ x\in E,\ t\in [0,\tau).
\end{align*}
Letting $n\rightarrow +\infty$ we obtain $CS_{n_{\tau}}(t)x=S_{n_{\tau}}(t)Cx,$ $x\in E,$ $t\in [0,\tau).$
Now we will prove that the condition (B) hold with the number $\alpha$ replaced with the number $n_{\tau}$
therein. By Proposition \ref{isto}(iii), we have $(\mathcal{G}(\varphi)x,\mathcal{G}(-\varphi^{\prime})x-\varphi(0)Cx)\in {\mathcal A},$ $\varphi \in {\mathcal D},$ $x\in E.$ Applying integration by parts and multiplying with $(-1)^{n_{\tau}+1}$ after that, the above implies
\begin{align*}
\Biggl( \int^{\infty}_{0}\varphi^{(n_{\tau}+1)}(s)\int^{s}_{0}S_{n_{\tau}}(r)x\, dr \, ds ,
\int^{\infty}_{0}\varphi^{(n_{\tau}+1)}(s)S_{n_{\tau}}(s)x \, ds +(-1)^{n_{\tau}}\varphi(0)Cx
\Biggr) \in {\mathcal A},
\end{align*}
for any $\varphi \in {\mathcal D}_{(-\tau ,\tau)}$ and $
x\in E.$ Plugging $ \varphi=I^{n_{\tau}+1}(\rho_{n}^{t})$ in this expression, we get that
\begin{align}\label{cor-tutf}
\Biggl( \int^{\infty}_{0}\rho_{n}^{t}(s)\int^{s}_{0}S_{n_{\tau}}(r)x\, dr \, ds ,
\int^{\infty}_{0}\rho_{n}^{t}(s)S_{n_{\tau}}(s)x \, ds +(-1)^{n_{\tau}}I^{n_{\tau}+1}(\rho_{n}^{t})(0)Cx
\Biggr) \in {\mathcal A},
\end{align}
for any $t\in [0,\tau)$ and $
x\in E.$ Let us prove that
\begin{align}\label{cor-tutf-prim}
\lim_{n\rightarrow +\infty}I^{n_{\tau}+1}(\rho_{n}^{t})(x)=(-1)^{n_{\tau}+1}g_{n_{\tau}+1}(t-x),\quad t\in [0,\tau),\ 0\leq x\leq t.
\end{align}
Let $t\in [0,\tau)$ and $x\in [0,t]$ be fixed. Then
a straightforward integral computation shows that
\begin{align*}
I^{n_{\tau}+1}(\varphi)(x)=(-1)^{n_{\tau}+1}\int^{\infty}_{x}\int^{\infty}_{x_{n_{\tau}}}\int^{\infty}_{x_{n_{\tau}-1}}\cdot \cdot \cdot \int^{\infty}_{x_{2}}\varphi(x_{1})\, dx_{1}\, dx_{2}\cdot \cdot \cdot \, dx_{n_{\tau}+1}
\end{align*}
for any $\varphi \in {\mathcal D}.$ For $ \varphi=I^{n_{\tau}+1}(\rho_{n}^{t}),$ we have
\begin{align*}
I^{n_{\tau}+1}\bigl(\rho_{n}^{t}\bigr)(0)&= (-1)^{n_{\tau}+1}\int^{t+(1/n)}_{x} \int^{t+(1/n)}_{x_{n_{\tau}}}\int^{t+(1/n)}_{x_{n_{\tau}-1}}\cdot \cdot \cdot \int^{t+(1/n)}_{x_{2}}
\\ & \times \rho_{n}^{t}(x_{1})\, dx_{1}\, dx_{2}\cdot \cdot \cdot \, dx_{n_{\tau}+1}
\\ & =(-1)^{n_{\tau}+1}\int^{t+(1/n)}_{x} \int^{t+(1/n)}_{x_{n_{\tau}}}\int^{t+(1/n)}_{x_{n_{\tau}-1}}\cdot \cdot \cdot \int^{t+(1/n)}_{x_{3}}
\\ & \times \Biggl[ 1-\int^{nx_{2}-nt}_{0}\rho(x_{1}) \, dx_{1}\Biggr]\, dx_{2}\cdot \cdot \cdot \, dx_{n_{\tau}+1}
\\ &=(-1)^{n_{\tau}+1}\int^{t+(1/n)}_{x} \int^{t+(1/n)}_{x_{n_{\tau}}}\int^{t+(1/n)}_{x_{n_{\tau}-1}}\cdot \cdot \cdot \int^{t+(1/n)}_{x_{3}}
\\ & \times \, dx_{2}\cdot \cdot \cdot \, dx_{n_{\tau}+1}
\\ &
-(-1)^{n_{\tau}+1}\int^{t+(1/n)}_{x} \int^{t+(1/n)}_{x_{n_{\tau}}}\int^{t+(1/n)}_{x_{n_{\tau}-1}}\cdot \cdot \cdot \int^{t+(1/n)}_{t}
\\ & \times \int^{nx_{2}-nt}_{0}\rho(x_{1}) \, dx_{1} \, dx_{2}\cdot \cdot \cdot \, dx_{n_{\tau}+1}
\\ & :=(-1)^{n_{\tau}+1}\bigl[I_{1}(t,x)-I_{2}(t,x)\bigr],\quad t\in [0,\tau).
\end{align*}
Since
\begin{align*}
\int^{t+(1/n)}_{t} \int^{nx_{2}-nt}_{0}\rho(x_{1}) \, dx_{1}
\, dx_{2}\leq 1/n,\quad t\in [0,\tau),\ n\in {\mathbb N},
\end{align*}
we have that $\lim_{n\rightarrow +\infty}I_{2}(t,x)=0,$
$t\in [0,\tau).$
Clearly,
\begin{align*}
\lim_{n\rightarrow +\infty}I_{1}(t,x)=\int^{t}_{x} \int^{t}_{x_{n_{\tau}}}\int^{t}_{x_{n_{\tau}-1}}\cdot \cdot \cdot \int^{t}_{x_{3}}\, dx_{2}\cdot \cdot \cdot \, dx_{n_{\tau}+1}=g_{n_{\tau}+1}(t-x).
\end{align*}
This gives (\ref{cor-tutf-prim}). Keeping in mind this equality and
letting $n\rightarrow +\infty$ in (\ref{cor-tutf}), we obtain (B). It remains to be proved the semigroup property of $(S_{n_{\tau}}(t))_{t\in [0,\tau)}.$ Toward this end, let us recall that
\begin{align}\label{astur-gilberto}
\bigl(\varphi*_0\psi\bigr)^{^{(n_\tau)}}(u)=\bigl(\varphi^{(n_\tau)}*_0\psi\bigr)(u)+\sum_{j=0}^{n_\tau-1}\varphi^{(j)}(0)\psi^{(n_\tau-1-j)}(u),\quad \varphi,\ \psi \in {\mathcal D},\ u\in {\mathbb R}.
\end{align}
Fix $x\in E$ and $t,\ s\in [0,\tau)$ with $t+s\in [0,\tau).$
Using (\ref{astur-gilberto}), (C.S.1) and the foregoing arguments, we get that, for every  $m,\ n\in {\mathbb N}$ sufficiently large:
\begin{align*}
&\int^{t}_{0}\int^{s}_{0}\rho_{n}^{t}(u)\rho_{m}^{s}(v)S_{n_{\tau}}(u)S_{n_{\tau}} (v)x\, du \, dv
\\ &=(-1)^{n_{\tau}}\int^{t+s}_{0} \Biggl[ \Bigl(\rho_{n}^{t}*_0 I^{n_{\tau}}(\rho_{m}^{s})\Bigr)(u)+\sum_{j=0}^{n_\tau-1}I^{n_{\tau}-j}(\rho_{n}^{t})(0)I^{j+1}(\rho_{m}^{s})(u)\Biggr]S_{n_{\tau}}(u)Cx\, du.
\end{align*}
Letting
$n\rightarrow +\infty,$ we obtain with the help of (\ref{cor-tutf-prim}) that
\begin{align*}
&\int^{s}_{0}\rho_{m}^{s}(v)S_{n_{\tau}}(t)S_{n_{\tau}} (v)x\, dv
\\ &=(-1)^{n_{\tau}}\lim_{n\rightarrow +\infty}\int^{t+s}_{0} \Biggl[ \Bigl(\rho_{n}^{t}*_0 I^{n_{\tau}}(\rho_{m}^{s})\Bigr)(u)
\\ &+\sum_{j=0}^{n_\tau-1}I^{n_{\tau}-j}(\rho_{n}^{t})(0)I^{j+1}(\rho_{m}^{s})(u)\Biggr]S_{n_{\tau}}(u)Cx\, du
\\ &=(-1)^{n_{\tau}}\int^{t}_{0}\Biggl[\sum_{j=0}^{n_{\tau-1}}(-1)^{n_{\tau}-j}g_{n_{\tau}-j}(t)I^{j+1}(\rho_{m}^{s})(u)\Biggr]S_{n_{\tau}}(u)Cx\, du
\\ & + (-1)^{n_{\tau}}\int^{t+s}_{t}\Biggl[ I^{n_{\tau}}(\rho_{m}^{s})(u-t)+\sum_{j=0}^{n_{\tau-1}}(-1)^{n_{\tau}-j}g_{n_{\tau}-j}(t)I^{j+1}(\rho_{m}^{s})(u)\Biggr]S_{n_{\tau}}(u)Cx\, du
\\ & =\sum_{j=0}^{n_{\tau-1}}(-1)^{j}g_{n_{\tau}-j}(t)\int^{s}_{0} I^{j+1}(\rho_{m}^{s})(u)S_{n_{\tau}}(u)Cx\, du
\\ & +(-1)^{n_{\tau}}\int^{t+s}_{t}I^{n_{\tau}}(\rho_{m}^{s})(u-t)S_{n_{\tau}}(u)Cx\, du.
\end{align*}
The semigroup property now easily follows by letting $m\rightarrow +\infty$ in the above expresion, with the help of (\ref{cor-tutf-prim}) and the identity
\begin{align*}
\sum_{j=0}^{n_{\tau-1}}g_{n_{\tau}-j}(t)g_{j+1}(s-u)=g_{n_{\tau}}(t+s-u),\quad u>0.
\end{align*}
Let $x\in{\mathcal N}({\mathcal G})$. Then there are $x_0,x_1,...,x_{n_{\tau}-1}\in E$, such that $S_{n_{\tau}}(t)x=\sum_{i=0}^{n_{\tau}-1}\frac{t^i}{i!}x_i$, for $t\in[0,\tau)$ and $x\in E$. For $\varphi\in{\mathcal D}$, such that $\varphi=1$ on a neighborhood of zero and integrating by parts $n_{\tau}$-times we have
$$T_ix={\mathcal G}({\varphi})x=(-1)^n\int\limits_0^{\infty}{\varphi}^{(n_{\tau})}(t)S_{n_{\tau}}(t)x\, dt=\varphi(0)\Big{(}S_{n_{\tau}}(t)\Big{)}x^{(n_{\tau}-1)}\, \big{|}_{t=0}=x_{n_{\tau}-1}.$$
Now, for $x$ is not an element in $\mbox{Ker}T_i$, $i=0,1,...,n_{\tau}-1$, $m\geq n_{\tau}-1$, we have that $x$ is not an element in $\mbox{Ker}S_{n_{\tau}}(t)$. But for $x\in\mbox{Ker}T_i$, $i=0,1,2,...,n_{\tau}-1$, we have that ${\mathcal G}({\varphi})x=0$ holds for all $\varphi\in{\mathcal D}_{(-\infty,\tau]}$ and this implies that $S_{n_{\tau}}(t)x=0$, $t\in[0,\tau)$.

\end{proof}

\begin{rem}\label{xcvbnm-prim}
\begin{itemize}
\item[(i)]
We have already seen that ${\mathcal G}(\cdot)\equiv 0$ is a degenerate pre-distribution semigroup with the generator ${\mathcal A}\equiv E\times E.$ Then, for every $\tau>0$ and for every number $n_{\tau}\in\mathbb{N},$ there exists only one local $n_{\tau}$-times integrated semigroup $(S_{n_{\tau}}(t)\equiv 0)_{t\in [0,\tau)}$ so that
(\ref{utf-88}) holds. It is clear that the condition (B) holds and that condition (A) does not hold here.
Denote by ${\mathcal A}_{\tau}$ the integral generator of $(S_{n_{\tau}}(t)\equiv 0)_{t\in [0,\tau)}.$
Then ${\mathcal A}_{\tau}=\{0\} \times E$ is strictly contained in the integral generator ${\mathcal A}$
of ${\mathcal G}.$
 Furthermore, if
$C\neq 0,$ then
there do not exist
$\tau>0$ and $n_{\tau}\in\mathbb{N}$ such that ${\mathcal A}$ is the integral generator (subgenerator) of a local
$n_{\tau}$-times integrated $C$-semigroup.
\item[(ii)] A similar line of reasoning as in the final part of the proof of \cite[Theorem 3.1.9]{knjigah} shows that
for each $(x,y)\in {\mathcal A}$ there exists elements $x_{0},x_{1},\cdot \cdot \cdot,x_{n_{\tau}}$ in $E$ such that
\begin{align*}
S_{n_{\tau}}(t)x-g_{n_{\tau}+1}(t)Cx-\int^{t}_{0}S_{n_{\tau}}(s)y\, ds=\sum \limits_{j=0}^{n_{\tau}}g_{j+1}(t)x_{j},\quad t\in [0,\tau)
\end{align*}
and $x_{j}\in {\mathcal A}x_{j-1}$ for $1\leq j\leq n_{\tau}.$ In purely multivalued case, it is not clear how we can prove that
$x_{j}=0$ for $0\leq j\leq n_{\tau}$ without imposing some additional unpleasant conditions.
\item[(iii)] Using dualization, we can simply reformulate the second equality appearing on the second line after the equation \cite[(11)]{ku112} in our context.
\end{itemize}
\end{rem}

The proof of subsequent theorem can be deduced by using the argumentation contained in the proof of
\cite[Theorem 3.1.8]{knjigah}.

\begin{thm}\label{lokal-int-C-prim}
Suppose that there exists a sequence $((p_k,\tau_k))_{k\in \mathbb{N}_{0}}$ in $\mathbb{N}_{0} \times (0,\infty)$
such that $\lim_{k\rightarrow \infty}\tau_{k}=\infty ,$ $(p_k)_{k\in \mathbb{N}_{0}}$ and
$(\tau_k)_{k\in \mathbb{N}_{0}}$ are strictly increasing,
as well as that for each
$k\in \mathbb{N}_{0}$
there exists a local $p_{k}$-times integrated
$C$-semigroup $(S_{p_{k}}(t))_{t\in [0,\tau_{k})}$ on $E$
so that
\begin{align}\label{zavezi-dot}
S_{p_{m}}(t)x=\bigl(g_{p_{m}-p_{k}}\ast_{0} S_{p_{k}}(\cdot)x \bigr)(t ),\quad x\in E,\ t\in [0,\tau_{k}),
\end{align}
provided $k<m.$
Define
$$
{\mathcal G}(\varphi)x:=(-1)^{p_{k}}\int \limits^{\infty}_{0}\varphi^{(p_{k})}(t)S_{p_{k}}(t)x\, dt,\quad \varphi \in {\mathcal D}_{(-\infty ,\tau_{k})},\ x\in E,\ k\in \mathbb{N}_{0}.
$$
Then $
{\mathcal G}$ is well-defined and $
{\mathcal G}$ is a pre-(C-DS).
\end{thm}

\begin{rem}\label{Banach}
\begin{itemize}
\item[(i)] Denote by ${\mathcal A}_{k}$ the integral generator of $(S_{p_{k}}(t))_{t\in [0,\tau_{k})}$ ($k\in \mathbb{N}_{0}$).
Then ${\mathcal A}_{k}\subseteq {\mathcal A}_{m}$ for $k>m$ and  $\bigcap_{{k\in {\mathbb N}_{0}}}{\mathcal A}_{k}\subseteq {\mathcal A},$ where ${\mathcal A}$ is the integral generator of ${\mathcal G}.$ Even in the case that $C=I,$ $\bigcup_{{k\in {\mathbb N}_{0}}}{\mathcal A}_{k}$ can be a proper subset of ${\mathcal A}.$
\item[(ii)] Suppose that
${\mathcal A}$ is a subgenerator of $(S_{p_{k}}(t))_{t\in [0,\tau_{k})}$ for all
$k\in \mathbb{N}_{0}.$ Then (\ref{zavezi-dot}) automatically holds.
\item[(iii)] In the case that $C=I,$ then it suffices to suppose that there exists an MLO ${\mathcal A}$ such that
${\mathcal A}$ is a subgenerator of a local $p$-times integrated
semigroup $(S_{p}(t))_{t\in [0,\tau)}$ for some $p\in {\mathbb N}$ and $\tau>0$  (\cite{catania}).
\end{itemize}
\end{rem}

Let $\alpha\in(0,\infty)\setminus \mathbb{N}$, $f\in\mathcal{S}$ and $n=\lceil\alpha\rceil$. Let us recall
that the Weyl fractional derivative $W^{\alpha}_+$ of order $\alpha$
is defined by
\begin{align*}
W^{\alpha}_+f(t):=\frac{(-1)^n}{\Gamma(n-\alpha)}\frac{d^n}{dt^n}\int\limits^{\infty}_t(s-t)^{n-\alpha-1}f(s)\,ds,
\;t\in\mathbb{R}.
\end{align*}
If $\alpha=n\in\mathbb{N}_{0}$, then we set $W^n_+:=(-1)^n\frac{d^n}{dt^n}.$
It is well known that the following equality holds: $W^{\alpha+\beta}_{+}f=W^{\alpha}_{+}W^{\beta}_{+}f$, $\alpha,\,\beta>0,$ $f\in\mathcal{S}$.

Suppose now that $\alpha\in(0,\infty)\setminus \mathbb{N}$ and
${\mathcal A}$ is the integral generator of a global $\alpha$-times integrated $C$-semigroup $(S_{\alpha}(t))_{t\geq 0}$ on $E.$ Then ${\mathcal A}$ is the integral generator of a global $n$-times integrated $C$-semigroup $(S_{n}(t))_{t\geq 0}$ on $E,$ where $n=\lceil \alpha \rceil$ and $S_{n}(t)x:=(g_{n-\alpha}\ast S_{\alpha}(\cdot)x)(t),$ $x\in E,$
$t\geq 0$ (\cite{catania}). Arguing as in \cite{C-ultra},
we have that:
$$
\int^{\infty}_0W^{\alpha}_+\varphi(t)S_{\alpha}(t)x\,dt=
(-1)^{n}  \int \limits^{\infty}_{0}\varphi^{(n)}(t)S_{n}(t)x\, dt,\quad x\in E,\ \varphi \in {\mathcal D}.
$$
Keeping in mind the proof of
\cite[Theorem 3.1.8]{knjigah}, we obtain the following:

\begin{thm}\label{miana}
Assume that $\alpha \geq 0$ and ${\mathcal A}$ is the integral generator of a global $\alpha$-times integrated $C$-semigroup $(S_{\alpha}(t))_{t\geq 0}$ on $E.$
Set
\begin{align*}
\mathcal{G}_{\alpha}(\varphi)x:=\int^{\infty}_0W^{\alpha}_+\varphi(t)S_{\alpha}(t)x\,dt,\quad x\in E,\ \varphi\in\mathcal{D}.
\end{align*}
Then ${\mathcal G}$ is a pre-(C-DS) whose integral generator contains ${\mathcal A}.$
\end{thm}

We will accept the following definition an exponential pre-(C-DS).

\begin{defn}\label{qdfn}
Let ${\mathcal G}$ be a pre-(C-DS).
Then $\mathcal{G}$ is said to be an exponential
pre-(C-DS) iff there exists $\omega\in\mathbb{R}$ such that $e^{-\omega t}\mathcal{G}
\in\mathcal{S}'(L(E))$.
We use the shorthand pre-(C-EDS)  to denote an exponential
pre-(C-DS).
\end{defn}

We have the following fundamental result:

\begin{thm}\label{miana-exp}
Assume that $\alpha \geq 0$ and ${\mathcal A}$ generates an exponentially equicontinuous $\alpha$-times integrated $C$-semigroup $(S_{\alpha}(t))_{t\geq 0}.$
Define ${\mathcal G}$ through $\mathcal{G}_{\alpha}(\varphi)x:=\int^{\infty}_0W^{\alpha}_+\varphi(t)S_{\alpha}(t)x\,dt,$ $ x\in E,$ $ \varphi\in\mathcal{D}$. Then $\mathcal{G}$ is a pre-(C-EDS) whose integral generator contains ${\mathcal A}.$
\end{thm}

\begin{rem}\label{s-t}
\begin{itemize}
\item[(i)]
Suppose that $\mathcal{G}$ is a pre-(C-EDS) generated by ${\mathcal A}$, $\omega\in\mathbb{R}$ and $e^{-\omega t}\mathcal{G}
\in\mathcal{S}'(L(E)).$ Suppose, further, that there exists a non-negative integer $n$ and a continuous function $V : {\mathbb R} \rightarrow L(E)$ satisfying that
\begin{align*}
\bigl \langle e^{-\omega t}\mathcal{G},\varphi \bigr \rangle =(-1)^{n}\int_{-\infty}^{\infty}\varphi^{(n)}(t)V(t)\, dt,\quad \varphi \in {\mathcal D},
\end{align*}
and that there exists a number $r\geq 0$ such that the operator family $\{(1+t^{r})^{-1}V(t): t\geq 0\}\subseteq L(E) $
is equicontinuous. Since $e^{-\omega \cdot}\mathcal{G}$ is a pre-(C-EDS) generated by ${\mathcal A}-\omega,$ the proof of Theorem \ref{lokal-int-C} shows that $(V(t))_{t\geq 0}$ is an exponentially equicontinuous
$n$-times integrated $C$-semigroup; by Theorem \ref{miana-exp}, the integral generator $\hat{{\mathcal A}}^{\omega}$ of $(V(t))_{t\geq 0}$
is contained in ${\mathcal A}-\omega.$ Define
$$
S_{n}(t)x:=e^{\omega t}V(t)x+\int \limits^{t}_{0} \sum
\limits_{k=1}^{\infty}\binom{n}{k}\frac{(-1)^{k}\omega^{k}(t-s)^{k-1}}{(k-1)!}
e^{\omega s}V(s)x\, ds.
$$
Arguing as in the proof of \cite[Theorem 2.5.1, Theorem 2.5.3]{knjigah}, we can prove that the MLO
$\hat{{\mathcal A}}^{\omega}+\omega \ ( \subseteq {\mathcal A})$ is the integral generator of an exponentially equicontinuous
$n$-times integrated $C$-semigroup $(S_{n}(t))_{t\geq 0}.$
\item[(ii)] The conclusions from Theorem \ref{miana-exp} and the first part of this remark can be reword for the classes of
$q$-exponentially equicontinuous
integrated $C$-semigroups and $q$-exponentially equicontinuous pre-(C-DS)'s; cf. \cite{C-ultra} for the notion.
\end{itemize}
\end{rem}

\begin{rem}\label{fundamentalna}
Suppose that ${\mathcal G}\in {\mathcal D}^{\prime}_{0}(L(E))$, ${\mathcal G}(\varphi)C=C{\mathcal G}(\varphi),$ $\varphi \in {\mathcal D}$
and ${\mathcal A}$ is a closed MLO on $E$ satisfying that $\mathcal{G}(\varphi){\mathcal A}\subseteq {\mathcal A}{\mathcal G}(\varphi),$ $\varphi \in {\mathcal D}$ and
\begin{equation}\label{dkenk}
\mathcal{G}\bigl(-\varphi'\bigr)x-\varphi(0)Cx\in {\mathcal A}\mathcal{G}(\varphi)x,\quad x\in E,\ \varphi \in {\mathcal D}.
\end{equation}
In \cite{C-ultra}, we have proved the following:
\begin{itemize}
\item[(i)] If ${\mathcal A}=A$ is single-valued, then ${\mathcal G}$ satisfies (C.S.1).
\item[(ii)] If ${\mathcal G}$ satisfies (C.S.2) holds, $C$ is injective and ${\mathcal A}=A$ is single-valued, then ${\mathcal G}$ is a (C-DS) generated by $C^{-1}AC.$
\item[(iii)] If $E$ is admissible and ${\mathcal A}=A$
is single-valued, then
the condition (C.S.2) automatically holds for ${\mathcal G}$.
\end{itemize}
As we have already seen, the conclusion from (ii) immediately implies that ${\mathcal A}=A$ must be single-valued and that
the operator $C$ must be injective.

Concerning the assertion (i), its validity is not true in multivalued case:
Let $C=I,$ let ${\mathcal A}\equiv E\times E,$ and let ${\mathcal G}\in {\mathcal D}^{\prime}_{0}(L(E))$  be arbitrarily chosen. Then ${\mathcal G}$ commutes with ${\mathcal A}$ and
 (\ref{dkenk}) holds but ${\mathcal G}$ need not satisfy (C.S.1).

Concerning the assertion (iii) in multivalued case, we can prove that the admissibility of state space $E$ implies that for each
$x\in {\mathcal N}({\mathcal G})$
there exist an integer $k\in {\mathbb N}$ and a finite sequence $(y_{i})_{0\leq i \leq k-1}$ in $D({\mathcal A})$ such that
$y_{i}\in {\mathcal A}y_{i+1}$ ($0\leq i \leq k-1$) and $Cx\in {\mathcal A}y_{0}\subseteq {\mathcal A}^{k+2}0.$
\end{rem}

Now we will reconsider some conditions introduced by J. L. Lions \cite{li121} in our new framework.
Suppose that $\mathcal{G}\in\mathcal{D}'_0(L(E))$  and ${\mathcal G}$ commutes with $C$. We analyze the following conditions for ${\mathcal G}$:
\begin{itemize}
\item[$(d_1)$] $\mathcal{G}(\varphi*\psi)C=\mathcal{G}(\varphi)\mathcal{G}(\psi)$, $\varphi,\,\psi\in\mathcal{D}_0$,
\item[$(d_3)$] $\mathcal{R}(\mathcal{G})$ is dense in $E$,
\item[$(d_4)$] for every $x\in\mathcal{R}(\mathcal{G})$, there exists a function $u_x\in C([0,\infty):E)$ so that
$u_x(0)=Cx$ and $\mathcal{G}(\varphi)x=\int_0^{\infty}\varphi(t)u_x(t)\,dt$, $\varphi\in\mathcal{D}$,
\item[$(d_5)$] $(Cx,\mathcal{G}(\psi)x)\in G(\psi_+)$, $\psi\in\mathcal{D},$ $x\in E$.
\end{itemize}
Suppose that $\mathcal{G}\in\mathcal{D}'_0(L(E))$  is a pre-(C-DS). Then it is clear that
$\mathcal{G}$
satisfies $(d_1),$ our previous considerations shows that $\mathcal{G}$
satisfies $(d_5);$ by the proof of \cite[Proposition 3.1.24]{knjigah}, we have that $\mathcal{G}$
also satisfies $(d_4).$ On the other hand, it is well known that $(d_1),$ $(d_4)$ and (C.S.2) taken together do not imply
(C.S.1), even in the case that $C=I;$ see e.g. \cite[Remark 3.1.20]{knjigah}. Furthermore, let $(d_1),$ $(d_3)$ and $(d_4)$ hold.
Then  $(d_5)$ holds, as well. In order to see  this, fix $x\in\mathcal{R}(\mathcal{G})$ and $\varphi\in\mathcal{D}$; then it suffices to show that
$(Cx,\mathcal{G}(\varphi)x)\in G(\varphi_+)$.  Suppose that $(\rho_n)$ is a regularizing sequence
and $u_x(t)$ is a function appearing in the formulation of the property $(d_4)$.
The arguments contained in the proof of \cite[Proposition 3.1.19]{knjigah} shows that, for every $\eta\in\mathcal{D}_{0}$, one has
\begin{align*}
\mathcal{G}(\rho_n)\mathcal{G}(\varphi_+*\eta)x&=\mathcal{G}((\varphi_+*\rho_n)*\eta)Cx
=\mathcal{G}(\eta)\mathcal{G}(\varphi_+*\rho_n)x\\
&=\mathcal{G}(\eta)\int\limits_0^{\infty}(\varphi_+*\rho_n)(t)u_x(t)\,dt
\\
& \to \mathcal{G}(\eta)\int_0^{\infty}\varphi(t)u_x(t)\,dt=\mathcal{G}(\eta)\mathcal{G}(\varphi)x,\;n\to\infty;\\
\mathcal{G}(\rho_n)\mathcal{G}(\varphi_+*\eta)x
&=\mathcal{G}(\varphi_+*\eta*\rho_n)Cx\to\mathcal{G}(\varphi_+*\eta)Cx,\;n\to\infty.
\end{align*}
Hence, $\mathcal{G}(\varphi_+*\eta)Cx=\mathcal{G}(\eta)\mathcal{G}(\varphi)x$ and $(d_5)$ holds, as claimed. On the other hand,
$(d_1)$ is a very simple consequence of $(d_5);$ to verify this, observe that for each $\varphi \in\mathcal{D}_0$ and $\psi \in\mathcal{D}$ we have $  \psi_{+}*\varphi=\psi \ast_0 \varphi =\varphi \ast_0 \psi,$ so that $(d_5)$ is equivalent to say that
$\mathcal{G}(\varphi \ast_0 \psi)C=\mathcal{G}(\varphi)\mathcal{G}(\psi)$ ($\varphi \in\mathcal{D}_0,$ $\psi \in\mathcal{D}$). In particular,
\begin{align}\label{line988}
\mathcal{G}(\varphi)\mathcal{G}(\psi)=\mathcal{G}(\psi)\mathcal{G}(\varphi),\quad \varphi \in\mathcal{D}_0,\ \psi \in\mathcal{D}.
\end{align}
Suppose now that $(d_5)$ holds. Let $\varphi \in\mathcal{D}_0$ and $\psi,\ \eta \in\mathcal{D}.$ Observing that $\psi_{+}\ast \eta_{+} \ast \varphi=(\psi \ast_{0}\eta)_{+} \ast \varphi$,  we have (cf. also \cite[Remark 3.13]{ku112}):
\begin{align}
\notag \mathcal{G}(\varphi)\mathcal{G}(\eta)\mathcal{G}(\psi)&= C\mathcal{G}(\eta_{+} \ast \varphi)\mathcal{G}(\psi)
\\\notag &=C\mathcal{G}(\psi_{+}\ast \eta_{+} \ast \varphi)=C\mathcal{G}\bigl( (\psi \ast_{0}\eta)_{+} \ast \varphi \bigr)C
\\\label{line98888} &=C\mathcal{G}(\varphi)\mathcal{G}( \psi \ast_{0}\eta )=\mathcal{G}(\varphi)\mathcal{G}( \psi \ast_{0}\eta )C.
\end{align}
By (\ref{line988})-(\ref{line98888}), we get
\begin{align}\label{line988888}
\mathcal{G}(\eta)\mathcal{G}(\psi)\mathcal{G}(\varphi)=\mathcal{G}( \psi \ast_{0}\eta )C\mathcal{G}(\varphi).
\end{align}
Due to (\ref{line988})-(\ref{line988888}), we have the following:
\begin{itemize}
\item[(i)] $(d_5)$ and $(d_3)$ together imply (C.S.1); in particular, $(d_1),$ $(d_3)$ and $(d_4)$ together imply  (C.S.1). This is an extension of \cite[Proposition 3.1.19]{knjigah}.
\item[(ii)] $(d_5)$ and $(d_2)$ together imply that ${\mathcal G}$ is a (C-DS); in particular, ${\mathcal A}=A$ must be single-valued and $C$ must be injective.
\end{itemize}
On the other hand, $(d_5)$ does not imply (C.S.1) even in the case that $C=I.$ A simple counterexample is $\mathcal{G}\in\mathcal{D}'_0(L(E))$
given by  ${\mathcal G}(\varphi)x:=\varphi(0)x,$ $x\in E,$ $\varphi \in {\mathcal D}$.

The exponential region $E(a,b)$ has been defined for the first time by W. Arendt, O. El--Mennaoui and V. Keyantuo in \cite{a22}:
$$
E(a,b):=\Bigl\{\lambda\in\mathbb{C}:\Re\lambda\geq b,\:|\Im\lambda|\leq e^{a\Re\lambda}\Bigr\} \ \ (a,\ b>0).
$$

Now we are able to state the following theorem:

\begin{thm}\label{herbs}
Let $a>0$, $b>0$ and $\alpha>0.$ Suppose that ${\mathcal A}$ is a closed \emph{MLO} and, for every $\lambda$ which belongs to the set $
E(a,b)
,$ there exists
an operator $F(\lambda)\in L(E)$ so that $F(\lambda){\mathcal A}\subseteq {\mathcal A} F(\lambda),$ $\lambda \in E(a,b),$ $F(\lambda)x\in (\lambda -{\mathcal A})^{-1}Cx,$ $\lambda \in  E(a,b),$ $x\in E,$
$F(\lambda)C=C F(\lambda),$ $\lambda \in  E(a,b),$
$F(\lambda)x-Cx=F(\lambda)y,$ whenever $\lambda \in  E(a,b)$ and $(x,y)\in {\mathcal A},$ and that the mapping $\lambda \mapsto F(\lambda)x$ is analytic on $\Omega_{a,b}$ and continuous on $\Gamma_{a,b},$
where $\Gamma_{a,b}$ denotes the upwards oriented boundary of $E(a,b)$
and $\Omega_{a,b}$ the open region which lies to the right of $\Gamma_{a,b}.$ Let the operator family $\{(1+|\lambda|)^{-\alpha}F(\lambda) : \lambda \in E(a,b)\}\subseteq L(E)$ be equicontinuous.
Set
\begin{align*}
\mathcal{G}(\varphi)x:=(-i)\int_{\Gamma_{a,b}}\hat{\varphi}(\lambda)F(\lambda)x\,d\lambda,
\;\;x\in E,\;\varphi\in\mathcal{D}.
\end{align*}
Then $\mathcal{G}$ is a pre-(C-DS) generated by an extension of  ${\mathcal A}. $
\end{thm}

\begin{proof}
Arguing as in non-degenerate case \cite{C-ultra}, we can prove with the help of Lemma \ref{integracija-tricky} that $\mathcal{G}\in {\mathcal D}^{\prime}_{0}(L(E))$ as well as that $\mathcal{G}$ commutes with $C$ and ${\mathcal A}.$ The prescribed assumptions imply by
\cite[Theorem 3.23]{catania} (cf. also \cite[Theorem 2.7.2(iv)]{knjigah}) that for each $n\in {\mathbb N}$ with $n>\alpha+1$ the MLO ${\mathcal A}$ subgenerates a local $n$-times integrated $C$-semigroup
$(S_{n}(t))_{t\in [0,a(n-\alpha-1))}.$ It is straightforward to prove \cite{C-ultra} that
$$
{\mathcal G}(\varphi)x=(-1)^{n}\int_{-\infty}^{\tau}\varphi^{(n)}(t)S_{n}
(t)x\, dt,\quad x\in E,\ \varphi \in {\mathcal D}_{(-\infty,a(n-\alpha-1))}.
$$
Now the conclusion directly follows from Theorem \ref{lokal-int-C-prim} and Remark \ref{Banach}(i)-(ii).
\end{proof}

\begin{rem}\label{filipa-1}
\begin{itemize}
\item[(i)]
If $C$ is injective, ${\mathcal A}=A$ is single-valued, $\rho_{C}(A) \subseteq E(a,b)$ and $F(\lambda)=(\lambda -{\mathcal A})^{-1}C,$ $\lambda \in E(a,b),$ then ${\mathcal G}$ is a (C-DS) generated by $C^{-1}AC$
(\cite{C-ultra}).
Even in the case that $C=I,$ the integral generator ${\mathcal A}$ of ${\mathcal G},$ in multivalued case, can strictly contain $C^{-1}{\mathcal A}C;$ see Remark \ref{xcvbnm-prim}(i).
\item[(ii)] Let ${\mathcal A}$ be a closed MLO, let $C$ be injective and commute with ${\mathcal A},$ and let $\rho_{C}({\mathcal A}) \subseteq E(a,b).$ Then the choice $F(\lambda)=(\lambda -{\mathcal A})^{-1}C,$ $\lambda \in E(a,b)$
is always possible; in this case, we have ${\mathcal A}0\subseteq
N({\mathcal G}(\varphi)),$ $\varphi \in {\mathcal D}$  (\cite{FKP}).
\end{itemize}
\end{rem}

Local integrated semigroups generated by multivalued linear operators (see e.g. \cite[Example 3.2.11(i)]{FKP}) can be used for construction of pre-(DS)'s.
In \cite[Theorem 3.2.21]{FKP} and \cite[Example 3.2.23]{FKP}, we have investigated the entire solutions
of backward heat Poisson equation, showing the existence of
an entire $C$-regularized semigroup ($C\in L(L^{p}(\Omega))$ non-injective) generated by the multivalued linear operator $\Delta \cdot m(x)^{-1}$ in $L^{p}(\Omega),$
where $\Omega$ is a bounded domain in ${\mathbb R}^{n}$.
This example can serve us to construct an important example of a pre-(C-DS); cf. also \cite[Example 3.24]{catania}. Examples of exponentially bounded integrated semigroups generated by multivalued linear operators can be found in \cite[Chapter II-III, Section 5.8]{faviniyagi} and these examples can be used for construction of exponential pre-(DS)'s. 
Also by Proposition \ref{kuki}(iii) the duals of non-dense pre-(C-DS)'s are pre-($C^{\ast}$-DS)'s on $E^{\ast}$, so this is another way of constructing of degenerate $C$-distribution semigroups.

By Proposition \ref{kuki}(iii), the duals of  non-dense
(C-DS)'s.


\begin{thebibliography}{99}

\bibitem{a22}
\textsc{W. Arendt, O. El--Mennaoui, V. Keyantuo},
\emph{Local integrated semigroups: evolution with jumps of regularity},
J. Math. Anal. Appl. \textbf{186} (1994), 572--595.

\bibitem{a43}
\textsc{W. Arendt, C.\,J.\,K. Batty, M. Hieber, F. Neubrander},
\emph{Vector-valued Laplace Transforms and Cauchy Problems},
Monographs in Mathematics {\bf 96}, Birkh\"auser/Springer Basel AG, Basel,
2001.

\bibitem{baskakov-chern}
\textsc{A. G. Baskakov, K. I. Chernyshov,}
\emph{On distribution semigroups
with a singularity at zero and bounded solutions of
differential inclusions,}
Math. Notes
\textbf{1} (2006), 19--33.

\bibitem{b42}
\textsc{R. Beals},
\emph{Semigroups and abstract Gevrey spaces},
J. Funct. Anal. \textbf{10} (1972), 300--308.

\bibitem{carol}
\textsc{R. W. Carroll, R. W. Showalter,}
\emph{Singular and
Degenerate Cauchy Problems,} Academic Press, New York, 1976.

\bibitem{cha}
\textsc{J. Chazarain},
\emph{Probl\'emes de Cauchy abstraites et applications\'a quelques probl\'emes mixtes},
J. Funct. Anal. \textbf{7} (1971), 386--446.

\bibitem{ci1}
\textsc{I. Cior\u anescu},
\emph{Beurling spaces of class $(M_p)$ and ultradistribution semi-groups},
Bull. Sci. Math. \textbf{102} (1978), 167--192.

\bibitem{cizi}
\textsc{I. Cioranescu, L. Zsido,}
\emph{$\omega$-Ultradistributions and Their
Applications to Operator Theory,} in: Spectral Theory, Banach Center
Publications \textbf{8}, Warsawza 1982, 77--220.

\bibitem{cross}
\textsc{R. Cross,}
{\it Multivalued Linear Operators,}
Marcel Dekker Inc., New York, 1998.

\bibitem{l1}
\textsc{R. deLaubenfels},
\emph{Existence Families, Functional Calculi and Evolution Equations},
Lect. Notes Math. \textbf{1570}, Springer, New York, 1994.

\bibitem{fat1}
\textsc{H. O. Fattorini},
\emph{The Cauchy Problem}, Addison-Wesley, 1983. MR84g:34003.

\bibitem{faviniyagi}
\textsc{A. Favini, A. Yagi,}
{\it Degenerate Differential Equations in Banach Spaces,}
Chapman and Hall/CRC
Pure and Applied Mathematics, New York, 1998.

\bibitem{ki90}
\textsc{J. Kisy\'nski},
\emph{Distribution semigroups and one parameter semigroups},
Bull. Polish Acad. Sci. \textbf{50} (2002), 189--216.


\bibitem{k91}
\textsc{H. Komatsu},
\emph{Ultradistributions, I. Structure theorems and a characterization},
J. Fac. Sci. Univ. Tokyo Sect. IA Math. \textbf{20} (1973), 25--105.

\bibitem{k911}
\textsc{H. Komatsu},
\emph{Ultradistributions, II. The kernel theorem and ultradistributions with support in a manifold},
J. Fac. Sci. Univ. Tokyo Sect. IA Math. \textbf{24} (1977), 607--628.

\bibitem{k82}
\textsc{H. Komatsu},
\emph{Ultradistributions, III. Vector valued ultradistributions. The theory of kernels},
J. Fac. Sci. Univ. Tokyo Sect. IA Math. \textbf{29} (1982), 653--718.

\bibitem{k92}
\textsc{H. Komatsu},
\emph{Operational calculus and semi-groups of operators},
in: Functional Analysis and Related topics (Kyoto),
Springer, Berlin, 213--234, 1991.

\bibitem{knjigah}
\textsc{M. Kosti\'c}, \emph{Generalized Semigroups and Cosine Functions},
Mathematical Institute SANU, Belgrade, 2011.

\bibitem{knjigaho}
\textsc{M. Kosti\'c}, \emph{Abstract Volterra Integro-Differential Equations},
Taylor and Francis Group/CRC Press/Science Publishers, Boca Raton, Fl., 2015.

\bibitem{FKP}
\textsc{M. Kosti\'c,}
{\it Abstract Degenerate Volterra Integro-Differential Equations: Linear Theory and Applications},
Book Manuscript, 2016.

\bibitem{C-ultra}
\textsc{M. Kosti\' c, S. Pilipovi\' c, D. Velinov,}
\emph{$C$-Distribution semigroups and $C$-ultradistribution semigroups in locally convex spaces,}
Siberian Math. J., accepted.

\bibitem{catania}
\textsc{M. Kosti\'c},
\emph{Degenerate $K$-convoluted $C$-semigroups and degenerate $K$-convoluted $C$-cosine functions in locally convex spaces,}
preprint.


\bibitem{kothe1}
\textsc{G.~K\" othe},
\emph{Topological Vector Spaces I}, Springer-Verlag, Berlin, Heidelberg, New York, 1969.



\bibitem{ku112}
\textsc{P.\,C. Kunstmann},
\emph{Distribution semigroups and abstract Cauchy problems},
Trans. Amer. Math. Soc. \textbf{351} (1999), 837--856.

\bibitem{ku113}
\textsc{P.\,C. Kunstmann},
\emph{Banach space valued ultradistributions and applications to abstract Cauchy problems},
preprint.

\bibitem{li121}
\textsc{J.\,L. Lions},
\emph{Semi-groupes distributions},
Port. Math. \textbf{19} (1960), 141--164.

\bibitem{isna-maiz}
\textsc{I. Maizurna,}
\emph{Semigroup Methods For Degenerate Cauchy Problems And Stochastic Evolution Equations,}
PhD Thesis, Univeristy of Adelaide, 1999.

\bibitem{martinez}
\textsc{C. Martinez, M. Sanz},
\emph{The Theory of Fractional Powers of Operators},
North--Holland Math. Stud. \textbf{187}, Elseiver, Amsterdam, 2001.

\bibitem{meise}
\textsc{R. Meise, D. Vogt},
\emph{Introduction to Functional Analysis},
Translated from the German by M.\,S. Ramanujan and revised by the authors.
Oxf. Grad. Texts Math., Clarendon Press, New York, 1997.

\bibitem{me152}
\textsc{I.\,V. Melnikova, A.\,I. Filinkov},
\emph{Abstract Cauchy Problems: Three Approaches},
Chapman Hall/CRC, Boca Raton, London, New York, Washington, 2001.

\bibitem{me153}
\textsc{I. V. Melnikova,}
\emph{The Cauchy problem for differential inclusion in Banach space and distribution spaces,}
Siberian Math. J. \textbf{42} (2001), 751--765.

\bibitem{me155}
\textsc{I. V. Melnikova, U. A. Anufrieva, V. Yu. Ushkov,}
\emph{Degenerate distribution semigroups and well-posedness of the Cauchy problem,}
Integral Transform Special Functions \textbf{6} (1998), 247--256.

\bibitem{pilip}
\textsc{S. Pilipovi\'c},
\emph{Tempered ultradistributions},
Boll. Un. Mat. Ital. \textbf{7} (1988), 235--251.

\bibitem{sch16}
\textsc{L. Schwartz},
\emph{Theorie des Distributions}, 2 vols.,
Hermann, Paris, 1950--1951.

\bibitem{1964}
\textsc{R. Shiraishi, Y. Hirata},
\emph{Convolution maps and semi-group distributions},
J. Sci. Hiroshima Univ. Ser. A-I \textbf{28} (1964), 71--88.

\bibitem{svir-fedorov}
\textsc{G. A. Sviridyuk, V. E. Fedorov,}
\emph{Linear Sobolev Type Equations and Degenerate Semigroups of Operators,}
Inverse and Ill-Posed Problems (Book {\bf 42}), VSP, Utrecht, Boston, 2003.

\bibitem{w241}
\textsc{S. Wang},
\emph{Quasi-distribution semigroups and integrated semigroups},
J. Funct. Anal. \textbf{146} (1997), 352--381.



\end{thebibliography}
\end{document}